\documentclass[10pt,a4paper]{amsart}

\usepackage{amssymb,amsmath,amsfonts}
\usepackage{amsthm}
\usepackage{enumerate}
\usepackage{theoremref}
\usepackage{color}
\usepackage[all]{xy}
\usepackage{graphicx}

\usepackage{mathpazo}
\usepackage[hyperindex,pageanchor]{hyperref}

\numberwithin{equation}{section}
\newtheorem{theorem}{Theorem}[section]
\newtheorem{lemma}[theorem]{Lemma}
\newtheorem{proposition}[theorem]{Proposition}
\newtheorem{corollary}[theorem]{Corollary}

\theoremstyle{definition}
\newtheorem{definition}[theorem]{Definition}
\theoremstyle{remark}
\newtheorem{remark}[theorem]{Remark}

\newtheorem{example}[theorem]{Example}
\newtheorem{question}[theorem]{Question}

\newcommand{\Ass}{\operatorname{Ass}}

\newcommand{\Att}{\operatorname{Att}}
\newcommand{\Ht}{\operatorname{ht}}
\newcommand{\pd}{\operatorname{pd}}

\newcommand{\Ext}{\operatorname{Ext}}

\newcommand{\Hom}{\operatorname{Hom}}

\newcommand{\depth}{\operatorname{depth}}

\newcommand{\cd}{\operatorname{cd}}

\newcommand{\vpl}{\operatornamewithlimits{\varprojlim}}

\newcommand{\vil}{\operatornamewithlimits{\varinjlim}}

\newcommand{\fm}{\mathfrak{m}}

\begin{document}

\title[Epimorphisms of local cohomology modules]{Epimorphisms of local cohomology modules, a general Peskine-Szpiro theorem, and an application to sheaf cohomology vanishing for thickenings}

\author[A. Dosea]{Andr\'e Dosea}
\author[M. Eghbali]{Majid Eghbali}
\author[C. B. Miranda-Neto]{Cleto B. Miranda-Neto}
\address{Departamento de Matem\'atica, Universidade Federal da Para\'iba - 58051-900, Jo\~ao Pessoa, PB, Brazil}
\email{andredosea@hotmail.com}
\address{School of Mathematics, Institute for Research in Fundamental Sciences (IPM), P. O. Box 19395-5746, Tehran, Iran}
\email{majideghbali83@gmail.com}

\address{Departamento de Matem\'atica, Universidade Federal da Para\'iba - 58051-900, Jo\~ao Pessoa, PB, Brazil}
\email{cleto@mat.ufpb.br}

\subjclass[2020]{Primary: 13D45, 13D07, 13A35, 14F17; Secondary: 14B15}

\keywords{Local cohomology, prime characteristic, sheaf cohomology of thickenings}

\begin{abstract} We study the surjectivity of certain maps involving local cohomology modules, which we can realize as a dual version of part of the investigation developed by Bhatt, Blickle, Lyubeznik, Singh and Zhang on the sheaf cohomology of thickenings (i.e., subschemes defined by powers of ideals), where injectivity played a central role. To this end, we introduce and investigate properties of cohomologically Mittag-Leffler (cML) rings, associated to a given flat local endomorphism (for instance the Frobenius map of a regular ring of prime characteristic), a class which we show to contain, in our setting, the so-called cohomologically full rings of Dao, De Stefani and Ma (in particular, Cohen-Macaulay, Stanley-Reisner, and Du Bois singularities) as well as rings with an ideal inducing a pure endomorphism of the quotient. Our two major specific goals rely upon the  prime characteristic setting. First, we extend for the class of cML rings a classical result of Peskine and Szpiro that relates the cohomological dimension and the height of a given Cohen-Macaulay ideal. Second, we prove and illustrate a Kodaira type vanishing result on the sheaf cohomology of thickenings. 
\end{abstract}

\maketitle

\section{Introduction}

 In this paper we are mainly motivated by work of Bhatt, Blickle, Lyubeznik, Singh and Zhang (see \cite{LP}) on the cohomology of thickenings , i.e., if $R$ is a standard graded polynomial ring over a field and $I$ is a homogeneous ideal of $R$, they investigate the sheaf cohomology  $H^i(X_t, \mathcal{O}_{X_t}(j))$ where $X_t$ is the closed subscheme of ${\rm Proj}(R)$ given by $$X_t={\rm Proj}(R/I^t) \quad \mbox{for} \quad t\geq 1.$$ In their investigation, the injectivity of certain maps involving (graded components) of Ext and local cohomology modules plays a central role.  

Here, we focus on the surjectivity of certain maps involving local cohomology modules and inverse limits of such modules. We work for simplicity in the local setting but after standard adaptations our theory and results are seen to hold as well in the graded case, which will be fundamental to our application on the vanishing of cohomology of thickenings -- one of our main purposes here.

The technical core of our investigation is the study of such maps in order to extend a classical result of Peskine and Szpiro (see \cite{PS73}), concerning the interplay between the cohomological dimension and the height of an ideal in the Cohen-Macaulay setting (the result will be recalled in the sequel), to a much broader class of singularities to be introduced and investigated in details in this paper. 

Throughout this work, we adopt the convention that all rings are commutative, unitary and Noetherian.

Let $R$ be a ring, $I$ an ideal of $R$ and $M$ an $R$-module. Recall that the cohomological dimension of $M$ with respect to $I$ is defined as
$$
\cd(M,I) = \sup \{i \geq 0 \mid  H^i_I (M) \neq 0\}.
$$
The inequality $\cd(M,I) \leq \dim M$ holds in general and the equality occurs for example if $I$ is $\fm$-primary in case $(R, \fm)$ is local. 

When $R$ is a complete local domain, we can use the Hartshorne–Lichtenbaum vanishing theorem (\cite[Theorem 8.2.1]{BroSha}) to ensure that
$\cd(R, I) \leq \dim (R) - 1$ if and only if $I$ is not $\fm$-primary. Huneke and Lyubeznik (\cite{HunLyu90}) provide conditions for 
$\cd(R, I) \leq \dim R - 2$ , where $R$ is a regular local ring containing a field.

In their landmark paper, Peskine and Szpiro (\cite{PS73}) proved that whenever $R$ is a regular local ring containing a field of positive characteristic
and the ring $R/I$ is Cohen–Macaulay,
the local cohomology module $H^i_I (R)$ vanishes for all $i > \dim R - \depth R/I$.
In this case, $$\dim R - \depth R/I = \Ht I,$$ which means that $\cd (R,I)= \Ht I$ as  $H^{\Ht I}_I (R) \neq 0$ (in this paper, "${\rm ht}$" denotes height). In other words, by a standard terminology, $I$ is a cohomologically complete intersection ideal.

On the other hand, to the best of our knowledge, the equality
\begin{equation} \label{cdequality}
\cd(R, I) = \dim R - \depth R/I
\end{equation}
holds true in the following special cases:
\begin{enumerate}
    \item [(a)] Square-free monomial ideals in a polynomial ring (see \cite[Theorem 1(iv)]{Lyubeznik84}).
\item [(b)] $R$ is a regular local ring  equipped with a flat local endomorphism $R \rightarrow R$ inducing a pure ring endomorphism $R/I \rightarrow R/I$ (see \cite{EghbaliBoix2023}).
\item [(c)] Cohomologically full rings with small depth and under some extra ring-theoretic conditions  (see \cite[Proposition 2.6]{DDM21}).
\end{enumerate}
   
Our first aim here is to generalize all such cases by means of a big family of singularities, dubbed {\it cohomologically Mittag-Leffler rings} herein, which we will show to contain, in our setting (namely, $R$ is either a regular local ring or a standard graded polynomial ring over a field, fitting into a cofinal flat triple in the sense of Definition \ref{cft}), the class of cohomologically full rings introduced by Dao, Stefani and Ma (which in turn is known to contain the Cohen-Macaulay, Stanley-Reisner, and Du Bois singularities), see Theorem \ref{cful}, and also the rings $R$ possessing an ideal $I$ inducing a pure map $R/I\rightarrow  R/I$; see Theorem \ref{cci-pure}. This will provide us a more general environment for the equality (\ref{cdequality}), thus yielding a generalization (see Theorem \ref{cddepth}) of the classical Peskine-Szpiro theorem originally stated in the Cohen-Macaulay case. 

Furthermore, we want to consider conditions under which the property (\ref{cdequality}) is transferred from one ideal to another; first in Proposition \ref{depth2} we study the invariance of depth and hence (by Theorem \ref{cddepth}) the invariance of (\ref{cdequality}). We then prove Theorem \ref{surj.lc} concerning the transference of the cML property, and in Corollary \ref{Ffull} we connect the cML property to the F-full property of Bhatt, Blickle, Lyubeznik, Singh and Zhang (see \cite{LP}). Along the way, our Corollary \ref{depth} generalizes Dao, De Stefani and Ma \cite[Corollary 2.8]{DDM21} on the interplay between projective dimension and minimal number of generators. 
We also devote a subsection, for completeness, to establishing some regularity criteria in prime characteristic, as the
 background of the cML concept is essentially a cofinal flat triple $(R, I, \varphi)$, where $R$ is assumed to be regular. Here we take $\varphi$ as the Frobenius map of $R$.

In the last section, we provide an application (see Corollary \ref{mainapp}) concerning the vanishing of sheaf cohomology for the thickenings of a Cohen-Macaulay (hence cohomologically Mittag-Leffler; see Theorem (\ref{cmCMLR})) closed subscheme of a projective space over a field of positive characteristic, mostly motivated by \cite{LP} as mentioned in the first paragraph. Several examples will be given.

So, in essence, this paper is about the  class of cohomologically Mittag-Leffler rings -- as a generalization of some well justified classes of rings -- and some of its applications.

\section{Preliminary Remarks}\label{Notation and Remarks}

In this section, we state some preliminary results and definitions that will be used along the paper. We fix a ring $R$ together with and ideal $I$ of $R$ and an $R$-module $M$.

\subsection{Formal local cohomology}

Let $(R, \fm)$ be a local ring. For each $i$, the projective limit $\displaystyle\vpl_t H^{i}_{\fm}(M/I^tM)$ is called the \textit{$i$-th formal local cohomology of $M$ with
respect to $I$}. Here, the projective system is induced by the natural epimorphisms $M/I^{t+1}M \to M/I^tM$. The formal local cohomology has been a source of investigation of many authors (see, e.g., \cite{formal}).

We will concentrate our attention on the case where $M=R$. Let us assume in addition that $R$ is an $n$-dimensional Cohen-Macaulay local ring possessing a canonical module $\omega_R$.
There is a duality relation between the $i$-th formal local cohomology module $\displaystyle\vpl_t H^{i}_{\fm}(R/I^t)$ and $H^{n-i}_I (R)$. Even more generally, this relation can be described in terms of any decreasing chain of ideals $\{I_t\}_{t \geq 0}$ cofinal with $\{I^t\}_{t \geq 0}$. Indeed, for any such chain, there is an isomorphism 
\begin{equation}\label{iso1}
H^{n-i}_I (\omega_R) \cong \left(\displaystyle\vpl_t H^{i}_{\fm}(R/I_t)\right)^{\vee}, 
\end{equation} where $(-)^{\vee}$ stands for Matlis dual. The reason why this is true comes from the fact that a local cohomology module can be alternatively expressed as
\begin{equation}\label{lc}
H^{i}_{I}(M)\cong \displaystyle\vil_t \Ext^{i}_R (R/I_t, M)
\end{equation}
together with local duality. Indeed,
\begin{equation}\label{isoauxiliar}
H_I^{n-i} (\omega_R) \cong \displaystyle\vil_t \Ext_R^{n-i} (R/I_t, \omega_R)
 \cong \displaystyle\vil_t  (H^{i}_{\fm}(R/I_t))^{\vee} \cong \left(\displaystyle\vpl_t H^{i}_{\fm}(R/I_t)\right)^{\vee},
\end{equation}
where, in the last isomorphism, we use the fact that contravariant $\Hom$ transforms inverse limits into direct limits.
As a byproduct, (\ref{iso1}) shows that
\begin{equation}\label{iso2}
\displaystyle\vpl_t H^{i}_{\fm}(R/I^t) \cong \displaystyle\vpl_t H^{i}_{\fm}(R/I_t).    
\end{equation}
We can read this as the version of (\ref{lc}) for formal local cohomology.
\begin{remark}\label{rmkisoauxiliar}
    With the same notation, we can use similar arguments to get an isomorphism
    $$
    \left(H^{n-i}_I (\omega_R)\right)^{\vee} \cong \displaystyle\vpl_t H^{i}_{\fm}(R/I_t).
    $$
    Indeed, using (\ref{isoauxiliar}), we can write
    $$
    \left(H^{n-i}_I (\omega_R)\right)^{\vee} \cong \left( \displaystyle\vil_t  (H^{i}_{\fm}(R/I_t)^{\vee} \right)^{\vee} \cong \displaystyle\vpl_t H^{i}_{\fm}(R/I_t)^{\vee \vee} \cong \displaystyle\vpl_t H^{i}_{\fm}(R/I_t).
    $$
\end{remark}
\subsection{Local endomorphisms and a bimodule structure}
Our set up here is a local ring $(R, \fm)$ and  a local endomorphism $\varphi: R \to R$.
 We start by attaching to $R$ an $(R,R)$-bimodule structure by means of $\varphi$. In explicit terms,
for any $r, r_1, r_2 \in R$, we put
$$r_1.(\varphi_{\ast}r).r_2:=\varphi_{\ast}(\varphi (r_1)rr_2).$$ With such a bimodule structure, $R$ will be denoted  $\varphi_{\ast}R$.

Now, let $\Phi$ be the functor on the category of $R$-modules given by 
\begin{equation}\label{functorPhi}
\Phi(M)=\varphi_{\ast} R \otimes_{R} M.
\end{equation}
The iteration $\Phi^t$ is the functor
$$\Phi^t(M)=\varphi_{\ast} R \otimes_R \Phi^{t-1}M,\quad  t \geq 1,$$
where $\Phi^0$ is interpreted as the identity functor. In particular, $\Phi^t(M)=\varphi^t_{\ast} R \otimes_R M.$
The key property here is the fact that the flatness of $\varphi$ is equivalent to the exactness of $\Phi$, as it can be easily verified. 
Under this hypotheses, we summarize some further consequences we shall use along the text.
\begin{itemize}
    \item[(a)]We have a natural isomorphism $\Phi(R) \cong R$ given by $\varphi_{\ast}r' \otimes r \mapsto \varphi (r)r'$.
    \item[(b)] If $M$ and $N$ are $R$-modules, then, as shown in \cite[2.6.1]{SinghWalther2007},  there are natural isomorphisms
\begin{equation}\label{isoext}
\Phi(\Ext^i_R(M,N)) \cong \Ext^i_R(\Phi(M), \Phi(N)), \quad \mbox{for\, all} \quad i \geq 0.
\end{equation}
 \item[(c)] Assuming in addition that $\{\varphi^t(I)R\}_{t \geq 0}$ forms a decreasing chain of ideals cofinal with $\{I^t\}_{t \geq 0}$, we can use \cite[p.\,291]{SinghWalther2007} to obtain isomorphisms
\begin{equation}\label{Phicohomology}
\Phi(H^i_{I}(R)) \cong H^i_{\varphi(I)R}(R)\cong H^i_{I}(R), \quad \mbox{for\, all} \quad i \geq 0.
\end{equation} 
\end{itemize}

\section{Cohomologically Mittag-Leffler rings}\label{sec3}
In this section we introduce the main object of this work, the $\varphi$-\textit{cohomologically Mittag-Leffler rings} ($\varphi$-cML rings, for short). We first discuss the relation between the $\varphi$-cML property and the surjectivity of the natural maps
$$
H^i_{\fm} (R/I^{t+1}) \to H^i_{\fm} (R/I^t). 
$$
The stability of this concept under $\fm$-adic completion is proved in this section, and a relation with the set of associated primes of local cohomology modules is also presented. We also show that the class of $\varphi$-cML rings contains some important classes of rings such as Stanley-Reisner rings, some pure rings, and cohomologically full rings.

Throughout this section, $R$ is either a regular local ring or a standard graded polynomial ring (over a field) with irrelevant ideal $\fm$, and $I$ is an ideal of $R$ (homogeneous if $R$ is graded). Moreover, $\varphi:R \to R$ denotes either a local ring morphism (in the local case) or a graded morphism (in the graded setting).
\subsection{Definition and basic properties}

We start with the notion of cofinal flat triples.

\begin{definition}\label{cft}
 We call $(R,I,\varphi)$ a \textit{cofinal flat triple} if $\varphi$ is a flat morphism and $\{\varphi^t(I)R\}_{t \geq 0}$ is a decreasing chain of ideals cofinal with $\{I^t\}_{t \geq 0}$.   
\end{definition}
\begin{example}\label{char p example}
    The archetypal example of a cofinal flat triple relies in positive characteristic. More precisely, suppose that $R$ has  prime characteristic $p>0$. If $F: R \to R$ stands for the Frobenius endomorphism (the map $x \mapsto x^p$), then a classical result due to Kunz (\cite[Theorem 2.1 and Corollary 2.7]{Kunz1969}) establishes that $F$ is a flat local endomorphism. Moreover, for any ideal $I$ of $R$, $F^e (I)R = I^{[p^e]}$ is a Frobenius power of $I$. In particular, given $t\geq 0$, $F^e (I)R \subset I^t$ for $e \gg0$ and $$I^{\mu(I)p^e} \subset F^e (I)R$$ where $\mu(-)$ denotes minimal number of generators.
    Therefore, $(R,I, F)$ is a cofinal flat triple.
\end{example}
\begin{example}\label{monomial example}
    Let $R$ be the polynomial ring $k[x_1, \ldots,x_n]$ and consider the $k$-linear endomorphism $\varphi: R \to R$, $x_i \mapsto x_i^d$ for a fixed positive integer $d$.
    As $R$ is a free module over $\varphi(R)$, $\varphi$ is a flat graded endomorphism. Moreover, for every square-free monomial ideal $I$, it is easy to check that $\{\varphi^t(I)R\}_{t \geq 0}$ is a descending chain of ideals cofinal with $\{I^t\}_{t \geq 0}$. Therefore, $(R, I, \varphi)$ is a cofinal flat triple.
    
\end{example}
\begin{definition}
Let $(R,I, \varphi)$ be a cofinal flat triple.
 The quotient ring $R/I$ is dubbed \textit{$\varphi$-cohomologically Mittag-Leffler} (abbreviated by $\varphi$-cML) if, for each $i$, the natural map
\begin{equation}\label{CMLR}
\displaystyle\vpl_t H^{i}_{\fm}(R/\varphi^t(I)R) \rightarrow H^{i}_{\fm}(R/\varphi^{s}(I)R)
\end{equation}
is an epimorphism for every $s$.
\end{definition}
For convenience we sometimes omit the map $\varphi$ from the notation by writing simply cML.
Quite generally, when all maps $M_j \to M_i$ in an inverse system of modules are epimorphisms, each of the natural projections $ \displaystyle\vpl_t M_t \to M_s$ is an epimorphism as well (see \cite[Lemma 3.5.3]{We}).
In particular, $R/I$ is $\varphi$-cML provided that, for each $i$, the natural map 
$$
H^i_{\fm} (R/\varphi^t(I)R) \to H^i_{\fm} (R/\varphi^s (I)R)
$$
is an epimorphism for all $s,t$.

On the other hand, the weaker hypothesis of the surjectivity of all maps
$$
H^i_{\fm} (R/I^t) \to H^i_{\fm} (R/I^s)
$$
is already sufficient to ensure that $R/I$ is $\varphi$-cML, as we observe in the next proposition.

\begin{proposition}
Let $(R,I,\varphi)$ be a cofinal flat triple. If, for each $i$, the natural map
\begin{equation}\label{1}
H^{i}_{\fm}(R/I^{t+1}) \rightarrow H^{i}_{\fm}(R/I^t),
\end{equation}
is an epimorphim for all $t$, then $R/I$ is $\varphi$-cML.
\end{proposition}
\begin{proof}
By the dimension formula,
$$
\dim R + \dim R/ \varphi(\fm)R = \dim R.
$$
Therefore, $\varphi(\fm)R$ is $\fm$-primary. Now, let $\Phi$ be the functor as in (\ref{functorPhi}). After applying the exact functor $\Phi^t(-)$ to (\ref{1}) we deduce that, for each $i$, the map
$$
H^{i}_{\fm}(R/\varphi^{t+1}(I)R) \rightarrow H^{i}_{\fm}(R/\varphi^t(I)R)
$$
is an epimorphism for every $t$.
Hence, using \cite[Lemma 3.5.3]{We} once again, we conclude that $R/I$ is $\varphi$-cML.
\end{proof}

Next, we study the invariance of the $\varphi$-cML property with respect to $\fm$-adic completion. 

\begin{proposition}\label{completionprop} \label{completion} Let $(R,I, \varphi)$ be a cofinal flat triple. Then $R/I$ is  $\varphi$-cML if and only if $\widehat{R}/I\widehat{R}$ is $\widehat{\varphi}$-cML.
\end{proposition}
\begin{proof} 
Notice that $$\widehat{\varphi}:\widehat{R} \longrightarrow \widehat{R},\ \widehat{\varphi}(\{r_t+\fm^t\}_t)= \{\varphi(r_t) +\varphi(\fm)^tR\}_t$$
is a flat local endomorphism since we have
$$\widehat{\varphi(R)}=\displaystyle\vpl_t  \varphi(R)/\fm^t \varphi(R)=\displaystyle\vpl_t  \varphi(R)/ \varphi^t(\fm)R=\displaystyle\vpl_t  \varphi(R/\fm^t)=\widehat{\varphi}(\widehat{R}).$$
The last equality follows by definition of $\widehat{\varphi}$. So by \cite[Exersise 7.1]{Matsumura87}, we conclude
that $\widehat{\varphi(R)}$ is a flat $\widehat{R}$-module. Also, it is easy to see that $\{\widehat{\varphi}^t(\widehat{I})\widehat{R}\}_t$ is a decreasing
chain of ideals cofinal with $\{\widehat{I}^t\}_t$. Now the claim follows from \cite[Proposition 3.3]{Schenzel2007}.
\end{proof}

Recall that if $(R,I, \varphi)$ is a cofinal flat friple such that $R/I$ is  $\varphi$-cML, then, in the light of Remark (\ref{rmkisoauxiliar}), the surjectivity (for each $i$) of 
$$\displaystyle\vpl_t H^{i}_{\fm}(R/\varphi^t(I)R) \rightarrow H^{i}_{\fm}(R/\varphi^{t}(I)R) \quad \mbox{for\, all} \quad t$$
is equivalent to the the surjectivity (for each $i$) of 
\begin{equation}\label{main}
\left(H^{n-i}_I(R)\right)^{\vee} \rightarrow H^{i}_{\fm}(R/\varphi^{t}(I)R) \quad \mbox{for\, all} \quad t, 
\end{equation} where $n={\rm dim}\,R$.
From this equivalence, we shall see how we may derive some new results on the set of  associated primes of local cohomology modules.

At this point, the theory of minimal primary decomposition and attached primes comes to play. We must use the fact that $\Ass_RM = \Att_R(M^{\vee})$ for every $ R$-module $M$. Further details can be found for instance in \cite{Hellus2007}.

\begin{theorem} \label{Ass}
Let $(R,I, \varphi)$ be a cofinal flat triple. Suppose that $R/I$ is  $\varphi$-cML. Then, for each $i$, $$\Ass_R (\Ext^i_R(R/\varphi^t(I)R,R)) \subseteq \Ass_R H^i_I(R) \quad \mbox{for\, all} \quad t.$$
\end{theorem}
\begin{proof}
Fix an integer $i$. By using the epimorphisms (\ref{main}) together with local duality, and in the light of \cite[Exercise 7.2.6]{BroSha}, we get 
$$\Ass_R (\Ext^i_R(R/\varphi^t(I)R,R))= \Att_R (H^{n-i}_{\fm}(R/\varphi^t(I)R)) \subseteq \Att_R({H^i_I(R)}^{\vee})=\Ass_R H^i_I(R),$$
for all $t$. Hence, the claim is proved.
\end{proof}
The finiteness of the set of associated primes of a local cohomology module is a classical problem in commutative algebra. Huneke and Sharp (see \cite[Corollary 2.3]{HunekeSharp}) proved that $\Ass_R H^i_I (R)$ is finite for every $i$ whenever $R$ is a regular ring of positive characteristic. Moreover, they proved an inclusion
$\Ass_R (\Ext^i_R(R/I,R)) \subseteq \Ass_R H^i_I(R)$.
In the next corollary, we derive (particularly from Theorem (\ref{Ass})) a situation where this inclusion becomes an equality.

\begin{corollary}
Let $(R,\fm)$ be a regular local ring of positive characteristic. If $R/I$ is $F$-pure $($see Definition (\ref{pure})$)$, then there is an equality $$\Ass_R (\Ext^i_R(R/I,R))= \Ass_R H^i_I(R),\ \text{\ for\ all\ }i.$$
\end{corollary}

\begin{proof}
 Since $R/I$ is $F$-pure, Theorem (\ref{cci-pure}) (to be proved in the next subsection) ensures that $R/I$ is  cML. Then we are done by Theorem (\ref{Ass}) and \cite[Corollary 2.3]{HunekeSharp}.
\end{proof}

\subsection{Two subclasses of the class of cML rings}
Our first example comes from the notion of pure maps, whose definition we now recall for completeness.

\begin{definition} \label{pure}
A ring homomorphism  $f:A \rightarrow B$ is pure if the map $f
\otimes 1:A \otimes_A M \rightarrow B \otimes_A M$  is injective for
each $A$-module $M$.  If $A$  contains a field of prime characteristic $p$, then $A$ is $F$-pure if the Frobenius endomorphism $F:A \rightarrow A$ is pure.
\end{definition}

\begin{remark} \label{pure2} 
 Consider the cofinal flat triple $(R,I,\varphi)$ as in Example (\ref{monomial example}). As $I$ is a square-free monomial ideal, $\varphi(I) \subset I$. So, we may consider the induced map
   $\bar{\varphi}:R/I \longrightarrow R/I$, which, by \cite[Example 2.2]{SinghWalther2007}, must be pure. 
\end{remark}
For more information, we refer the reader to \cite[Section 2]{EghbaliBoix2023}. Now we prove:

\begin{theorem} \label{cci-pure}
Let $(R,I,\varphi)$ be a cofinal flat triple. Suppose that $\varphi(I) \subset I$.
If the induced map $\bar{\varphi}:R/I \longrightarrow R/I$ is pure,
then $R/I$ is  $\varphi$-cML.
\end{theorem}

\begin{proof}
 
Fix an index $i$. From \cite[Theorem 2.8]{SinghWalther2007}, the natural map $$\Ext^i_R(R/\varphi^t(I)R,R) \rightarrow \Ext^i_R(R/\varphi^{t+1}(I)R,R)$$
is injective for all $t>0$.
By applying the Matlis duality functor to the monomorphism above, we obtain an epimorphism
$$\Ext_R^{i} (R/\varphi^{t+1}(I)R,R)^{\vee} \rightarrow \Ext_R^{i} (R/\varphi^{t}(I)R,R)^{\vee},$$
which therefore admits the shape
$$H^{\dim R-i}_{\fm}(R/\varphi^{t+1}(I)R) \rightarrow H^{\dim R-i}_{\fm}(R/\varphi^{t}(I)R).$$
Now we are done by \cite[Lemma 3.5.3]{We}.
\end{proof}
As a byproduct, we are able to show that Stanley-Reisner rings are cML.

\begin{corollary} \label{SR}\label{polynomials} Stanley--Reisner rings $($i.e., quotient of polynomial or power series rings over a field by a square-free monomial ideal$)$ are cML.
\end{corollary}

\begin{proof}
By Proposition (\ref{completion}), it suffices to prove the result in case  $R=k[x_1,\ldots, x_n]$ is a polynomial ring over a field $k$ with maximal ideal $\fm=(x_1,\ldots, x_n)$ and $I$ is a square-free monomial ideal. Now, the result is an immediate consequence of Remark (\ref{pure2}) together with Theorem (\ref{cci-pure}).
\end{proof}

Our next goal is to exhibit another class of rings which are cML. Motivated by a question raised by Eisenbud, Musta\textcommabelow{t}\u{a} and Stillman (\cite[Question 6.2]{EMM}), H. Dao, A. D. Stefani and L. Ma (\cite{DDM21}) introduced the following interesting concept.

\begin{definition} \label{cfr}
Given an integer $i$,
a local ring $(R, \fm)$ is said to be $i$-\textit{cohomologically full} if for every surjective map $(T, \mathfrak{n}) \to (R, \fm)$ such that $T$ and $R$ have the same characteristic and $R/ \sqrt{(0)} = T/ \sqrt{(0)}$, the induced map $$H^i_{\fm} (T) \to H^i_{\mathfrak{n}} (R)$$ is an epimorphism.
If $R$ is $i$-cohomologically full for all $i$, then $R$ is called \textit{cohomologically full}.
\end{definition}

\begin{theorem}\label{cful}
Let $(R,I, \varphi)$ be a cofinal flat triple. If $R/I$ is cohomologically full, then $R/I$ is $\varphi$-cML.    
\end{theorem}
\begin{proof}
By \cite[Proposition 2.1]{DDM21} and since $R/I$ is cohomologically full, the natural map
\begin{equation}\label{extmap}
    \Ext^i_R (R/I,R) \to H^i_I (R)
\end{equation}
is a monomorphism for all $i$.
Let $\Phi$ be the functor as defined in (\ref{functorPhi}). We have seen that $\Phi$ is exact since $\varphi$ is a flat morphism.
Hence, by applying $\Phi^s$ to (\ref{extmap}) and using (\ref{isoext}) along with (\ref{Phicohomology}), we get a monomorphism
\begin{equation}\label{extmono}
\Ext^i_R (R/ \varphi^s(I)R, R) \to H^i_I (R).
\end{equation}
As $\{\varphi^t(I)\}_t$ is cofinal with $\{I^t\}_t$, we can use (\ref{iso1}) to derive an isomorphism
\begin{equation}\label{iso3}
H_I^{n-i} (R) \cong \left(\displaystyle\vpl_t H^{i}_{\fm}(R/\varphi^t(I)R)\right)^{\vee}.
\end{equation}
Then, the map (\ref{extmono}) becomes
$$
\Ext^i_R (R/ \varphi^s(I)R, R) \to \left(\displaystyle\vpl_t H^{n-i}_{\fm}(R/\varphi^t(I)R)\right)^{\vee},
$$
where $n = \dim R$.
After dualizing the map above and taking into account the isomorphisms
\begin{align*}
\left(\displaystyle\vpl_t H^{n-i}_{\fm}(R/\varphi^t(I)R)\right)^{\vee \vee} &\cong \left(\displaystyle\vil_t  (H^{n-i}_{\fm}(R/\varphi^t(I)R)^{\vee}\right)^{\vee} \\ &\cong  
\displaystyle\vpl_t (H^{n-i}_{\fm}(R/\varphi^t(I)R)^{\vee \vee} \\ &\cong 
\displaystyle\vpl_t H^{n-i}_{\fm}(R/\varphi^t(I)R),
\end{align*}
we get the epimorphism
$$
\displaystyle\vpl_t H^{n-i}_{\fm}(R/\varphi^t(I)R \to H^{n-i}_{\fm} (R/ \varphi^s(I)R).
$$
Since this holds for all $i$ and all $s$, we are done. 
\end{proof}

All Cohen-Macaulay local rings are  cohomologically full by \cite[Remark 2.5]{DDM21}. So we get the following consequence (see Theorem \ref{cmCMLR} for another proof).

\begin{corollary}\label{CMcor}
   Let $(R,I, \varphi)$ be a cofinal flat triple. If $R/I$ is Cohen-Macaulay, then $R/I$ is $\varphi$-cML. 
\end{corollary}

Another important subclass of the class of cohomologically full rings is the one formed with the reduced rings that are essentially of finite type over the complex number field $\mathbb{C}$ with Du Bois singularities (see \cite[Remark 2.5]{DDM21}).

\begin{corollary}
   Let $(R,I, \varphi)$ be a cofinal flat triple, with $R/I$ essentially of finite type over $\mathbb{C}$. If $R/\sqrt{I}$ is Du Bois, then $R/ \sqrt{I}$ is $\varphi$-cML. 
\end{corollary}

\begin{remark} In the case where $(R, \fm,k)$ is a standard graded $k$-algebra, the definition of $R$ being cohomologically full is translated by the localization $R_{\fm}$ being cohomologically full (see \cite[Definition 2.4]{DDM21}). It turns out that Stanley-Reisner rings are cohomologically full (see \cite[Remark 2.5]{DDM21}).
In particular, Theorem (\ref{cful}) furnishes an alternative way (see also Corollary (\ref{SR})) to prove that Stanley-Reisner rings are cML. \end{remark}

So far we have seen that, for a cofinal flat triple $(R, I, \varphi)$, if the induced map $\overline{\varphi}: R/I \to R/I$ is pure or $R/I$ is cohomologically full, then $R/I$ is $\varphi$-cML. It might be possible that the union of such two classes constitute the entire class of cML rings. Thus, the following question seems to be natural.

\begin{question}
Are there examples of cML rings not coming from the ones given by Theorem (\ref{cci-pure}) and Theorem (\ref{cful})?    
\end{question}

\section{Cohomological dimension, and a generalization of the Peskine-Szpiro theorem}

As before, we let $R$ be either a regular local ring or a standard graded polynomial ring (over a field) with irrelevant ideal $\fm$, and $I$ is an ideal of $R$ (homogeneous if $R$ is graded). Finding out the relation between the cohomological dimension of $I$ in $R$ and the depth of the quotient ring $R/I$ is a long-standing problem in commutative algebra.
In this section, our main goal (see Theorem \ref{cddepth}) is to prove that the equality
$$
\cd(R, I) = \dim R - \depth R/I
$$
holds whenever $R/I$ is cML. Now recall that if $R/I$ is Cohen-Macaulay then $R/I$ is cML (see Corollary \ref{CMcor}). Thus, our result generalizes a celebrated result of Peskine-Szpiro which establishes the above equality when ($R$  contains a field of positive characteristic, and) $R/I$ is Cohen-Macaulay, in which case $\dim R - \depth R/I$ is simply the height of $I$.  Our theorem will also shed light on a characterization of the Cohen-Macaulay property for cML rings in terms of cohomologically complete intersection ideals. In particular, we shall see that if $R/I$ is an one-dimensional cML ring, then $R/I$ is, necessarily, Cohen-Macaulay.

\subsection{A generalized Peskine-Szpiro theorem}
The following lemma (see \cite[Lemma 3.4]{EABR}) will be used in the sequel.

\begin{lemma} \label{EABR} 
Let $(R, \fm)$ be a $($not necessarily regular$)$ local ring and $I \subset R$ an ideal of finite projective dimension. Assume that $\varphi:R \rightarrow R$ is a ring  endomorphism satisfying the going down property. Then, $\depth R/I \leq \depth R/\varphi(I)R$.
\end{lemma}

We state and prove our generalized Peskine-Szpiro theorem:

\begin{theorem} \label{cddepth}
Let $(R,I, \varphi)$ be a cofinal flat triple and suppose $R/I$ is $\varphi$-cML. If $\dim R = n$, then
$$n-\cd(R,I)=n-\cd(R,\varphi^t(I)R ) = \depth R/\varphi^t(I)R \quad \mbox{for\, all} \quad t,$$
and in particular,
$$n-\cd(R,I) = \depth R/I.$$
\end{theorem}

\begin{proof}
First, in view of (\ref{Phicohomology}) we have $$\cd (R,I)= \cd(R, \varphi^t(I)R) \quad \mbox{for\, all} \quad t.$$
On the other hand, by (\ref{completionprop}), we may suppose that $R$ is complete.
Let us start proving the inequality
$$
n- \cd(R, \varphi^t(I)R) \leq \depth R/ \varphi^t(I)R \quad \mbox{for\, all} \quad t.
$$
Recall the isomorphism (\ref{iso3})
$$
H_I^{n-i} (R) \cong \left(\displaystyle\vpl_t H^{i}_{\fm}(R/\varphi^t(I)R)\right)^{\vee}.
$$
Now suppose that $i < n - \cd(R,I)$. In particular, the modules above vanish.
Since the Matlis duality functor is faithful, we obtain
$$
\displaystyle\vpl_t H^{i}_{\fm}(R/\varphi^t(I)R) = 0.
$$
As $R/I$ is $\varphi$-cML, the epimorphism (\ref{CMLR}) implies the vanishing of $ H^{i}_{\fm}(R/\varphi^t(I)R)$ for all $t$.
 Hence, $$n-\cd(R,I )=n-\cd(R,\varphi^t(I)R ) \leq \depth R/\varphi^t(I)R \quad \mbox{for\, all} \quad t.$$ 
Now, we proceed to prove that 
$n - \cd(R,I) \geq \depth R/I.$
Since every flat homomorphism has the going down property, we may use Lemma (\ref{EABR}) to deduce 
$$\depth(R/I) \leq \depth(R/\varphi(I)R) \leq \depth(R/\varphi^2(I)R) \leq \cdots.$$
As a consequence, $H^i_{\fm}(R/\varphi^t(I)R) = 0$ for all $i < \depth R/I$ and all $t$. From (\ref{iso3}), we deduce that $H^{n-i}_I(R) = 0$, as desired.

We have shown $n - \cd(R,I) = \depth R/I \leq \depth R/\varphi^t(I)R$ for all $t$.
It remains to prove that $ \depth R/I = \depth R/\varphi^t(I)R$ for all $t$.

Suppose $\depth R/I < \depth R/ \varphi^{t_0}(I)R$ for some $t_0$.
Then, setting $i=\depth R/I$ in (\ref{iso3}), we obtain
$$
H^{\cd(R,I)}_I (R) \cong \left(\displaystyle\vpl_t H^{\depth R/I}_{\fm}(R/\varphi^t(I)R)\right)^{\vee}
$$
On the other hand, $$H^{\depth R/I}_{\fm}(R/\varphi^t(I)R) =0\quad \mbox{for\, all} \quad t \geq t_0.$$ Therefore, the inverse limit above vanishes as well and we get a contradiction.
\end{proof}
As a consequence, we derive an upper bound for the projective dimension of $R/I$ which, by virtue of Theorem \ref{cful}, generalizes \cite[Corollary 2.8]{DDM21}.

\begin{corollary} \label{depth}  Let $(R,I, \varphi)$ be a flat cofinal triple. Suppose that $R/I$ is  $\varphi$-cML. Then $$\pd R/I \leq \displaystyle\min_{t \geq 0} \mu(\varphi^t(I)R). $$ In particular, $\pd R/I \leq \mu(I)$.
\end{corollary}

\begin{proof}
 Using the Auslander-Buchsbaum equality and Theorem (\ref{cddepth}), we have
$$ \pd R/I = \dim R - \depth R/I= \cd(R,\varphi^t(I)R ) \leq \mu({\varphi^t(I)R})\quad \mbox{for\, all} \quad t.$$ Now the claim is proved.
\end{proof}

\begin{remark}
Notice that the same argument as above shows that $\pd R/I = \pd/ \varphi^t(I)R$ for all $t$.    
\end{remark} 
Now, we turn our attention to the following question:

\begin{question}\label{Q1}\rm When is a cML ring a Cohen-Macaulay ring?
    
\end{question}
The next result offers a first glance in this direction.
\begin{proposition} \label{S1}
Let $(R,I, \varphi)$ be a cofinal flat triple. Suppose that $R/I$ is  $\varphi$-cML.
\begin{enumerate}
\item[\rm(i)] If $R/I$ is Cohen-Macaulay, then $I$ is cohomologically complete intersection. 
\item [\rm(ii)] If $\dim R/I>0$, then $\depth R/I>0$. 
\item [\rm(iii)] If $\dim R/I=1$, then $R/I$ is Cohen-Macaulay and $\cd(R,I)= \dim R-1$.
\end{enumerate}
\end{proposition}
\begin{proof}
\begin{enumerate}
\item[\rm(i)] Since $\dim R/I = \depth R/I$ and $\dim R - \dim R/I = \Ht I$, the claim is an immediate consequence of Theorem (\ref{cddepth}). 
    \item[\rm (ii)] By Proposition (\ref{completionprop}), $\widehat{R}/ I \widehat{R}$ is   $\widehat{\varphi}$-cML.
    If $\depth R/I = 0$, then $\depth \widehat{R}/I \widehat{R} =0$ as well.
    Therefore, it follows from Theorem (\ref{cddepth}) that
    $n = \cd(\widehat{R}, I \widehat{R})$, where $n = \dim R = \dim \widehat{R}$.
    Since $\widehat{R}$ is a complete domain, we conclude that $I \widehat{R}$ is $\fm \widehat
    R$-primary. In particular, $\dim \widehat{R}/I \widehat{R}=0$, which is a contradiction.
    Hence, $\depth R/I >0$.
    \item[\rm (iii)] It follows directly from item (ii) together with Theorem (\ref{cddepth}). 
\end{enumerate}\end{proof}

\begin{remark}\label{myremark2}
If $(R,\fm)$ is a complete local domain with $\dim R=d$ and $I$ is an ideal of $R$, then as a consequence of the Hartshorne-Litchenbaum vanishing Theorem, $\dim R/I >0$ if and only if $\cd(R,I) \leq d-1$.
For one-dimensional ideals, Proposition (\ref{S1}) shows that being  cML is a sufficient condition on $R/I$ (even in the non-complete case) to ensure that $\cd(R,I)=d-1$.
\end{remark}
\begin{remark}\label{myremark3}
Proposition (\ref{S1})(ii) does not extend to higher values of $\dim R/I$. Indeed, for each $n>1$, we can consider the ideal $$I=(x_1x_2, \ldots, x_1 x_{n+1})$$ in $R=k[x_1, \ldots,x_{n+1}]$. Since $R/I$ is a Stanley-Reisner ring, it must be cML (see Corollary (\ref{polynomials})). On the other hand,
$\dim R/I =n$, $\cd(R,I)=n$, but $R/I$ is not Cohen-Macaulay.    
\end{remark}
The following result is what we need in order to answer Question (\ref{Q1}) completely.
\begin{proposition}\label{rmkCM}
 Let $(R,I, \varphi)$ be a cofinal flat triple $(R,I, \varphi)$. If $R/I$ is a Cohen-Macaulay, then so is $R/\varphi^t(I)R$ for all $t$. Moreover, $\dim R/I = \dim R/ \varphi^t(I)R$ for every $t$.  
\end{proposition}
\begin{proof}
Since $\{\varphi^t(I)R\}_{t\geq 0}$ is cofinal with $\{I^t\}_{t \geq 0}$, the inclusion $\sqrt{I}\subset \sqrt{\varphi^t(I)R}$ holds for every $t$. Therefore, by means of Lemma (\ref{EABR}) we have successive inequalities
$$\depth R/I \leq \depth R/\varphi^t(I)R \leq \dim R/\varphi^t(I)R \leq \dim R/I.$$
Thus, we are done as $R/I$ is Cohen-Macaulay.    
\end{proof}

\begin{theorem}\label{cmCMLR}
Let $(R, I, \varphi)$ be a  cofinal flat triple.
The following statements are equivalent:
\begin{enumerate}
    \item[\rm(i)] $R/I$ is Cohen-Macaulay.
     \item[\rm(ii)] $R/I$ is  $\varphi$-cML and $\Ht(I)=\cd(R,I)$.    
\end{enumerate}
Moreover, if we assume that $\varphi(I) \subset I$, then the statements above are equivalent to the following assertion:
\begin{enumerate}
    \item[\rm(iii)] $R/I$ is  $\varphi$-cML and $R/\varphi^t(I)R$ is Cohen-Macaulay for some $t$. 
\end{enumerate}
\end{theorem}
\begin{proof}
Let us denote $d= \dim R/I$.

    \rm (i) $\Longrightarrow \rm (ii)$: Suppose that $R/I$ is Cohen-Macaulay. From Corollary \ref{CMcor} we know $R/I$ is $\varphi$-cML, but we proceed to give a simple, self-independent proof of this fact. By Proposition (\ref{rmkCM}), the $R$-module $R/ \varphi^t(I)R$ is Cohen-Macaulay with dimension $d$ for all $t$.
  The short exact sequence
  $$0 \rightarrow \varphi^{t}(I)R/\varphi^{t+1}(I)R \rightarrow
   R/\varphi^{t+1}(I)R \rightarrow R/\varphi^{t}(I)R \rightarrow 0,$$ induces the following exact sequence in cohomology   
   $$H^d_{\fm}(R/\varphi^{t+1}(I)R) \stackrel{\phi^t}{\rightarrow}
   H^d_{\fm}(R/\varphi^t(I)R) \rightarrow H^{d+1}_{\fm} (\varphi^{t}(I)R/\varphi^{t+1}(I)R).$$
   Now, using Grothendieck's Vanishing Theorem we deduce that $\phi^t$ is an epimorphism for all $t$. Then, applying (\cite[Lemma 3.5.3]{We}), we see that $R/I$ is  $\varphi$-cML.
   
   Due to the Cohen-Macaulayness of $R$, $\Ht(I) = \dim R - d$.
On the other hand, since $R/I$ is  $\varphi$-cML, the equality
$$
\cd(R,I) = \dim R - \depth R/I
$$
holds true by Theorem (\ref{cddepth}). Now, the equality $\Ht I = \cd(R,I)$ is immediate since $R/I$ is Cohen-Macaulay.

\rm (ii) $\Longrightarrow \rm (i)$: The argument is similar to the one in the last part of the previous implication.

From what we have done along with Remark (\ref{rmkCM}), it is clear that (i) implies (iii).

For the converse, notice that the additional hypothesis  $\varphi(I)\subset I$ forces the ideals $I$ and $\varphi^t(I)R$ to have the same radical for all $t$.
In particular, $\dim R/ \varphi^t(I)R = d.$
By Theorem (\ref{cddepth}), $\depth R/I = \depth R / \varphi^t(I)R$, and we are done. \end{proof}


\begin{remark} \label{C-Mht=cd}
 It is worth mentioning that, under the assumptions of Theorem (\ref{cmCMLR}), from the Cohen-Macaulayness of $R/I$ it is possible to derive the equality $\Ht(I)=\cd(R,I)$ without making any use of the fact that $R/I$ is  $\varphi$-cML. First of all, as $\Ht(I) \leq \cd(R,I)$, it suffices to prove  $\cd(R,I) \leq \Ht(I)$. Notice that $H^i_{\fm}(R/I)=0$ for all $i < \dim R-\Ht(I)= \depth R/I$. Applying $\Phi^t(-)$  and taking the inverse limit, we derive that 
 $$\displaystyle\vpl_t H^{i}_{\fm}(R/\varphi^t(I)R)= 0 \quad \mbox{for\, all} \quad i < \dim R-\Ht(I).$$  Now, using (\ref{iso3}) we get $H_I^{\dim R-i} (R)=0$ for all $i < \dim R-\Ht(I)$, which implies $\cd(R,I) \leq \Ht(I)$, as desired.
\end{remark}

Now the following characterization is immediate.

\begin{corollary}
    Let $(R,I, \varphi)$ be a cofinal flat triple. Suppose that $R/I$ is $\varphi$-cML.
    Then, $R/I$ is Cohen-Macaulay if and only if $I$ is cohomologically complete intersection.
\end{corollary}

\subsection{Transference of the cML property}
We now turn to discussing when the property of being cML can be transferred from a quotient ring $R/I$ to another, say $R/J$. So we start raising the following question:

\begin{question} \label{CMLRradical}
   Suppose that $(R,I, \varphi)$ and $(R,J,\varphi)$ are cofinal flat triples, with $\sqrt{I}=\sqrt{J}$. If $R/I$ is cML, does the same hold for $R/J$? 
\end{question}

The interest in finding (at least partial) answers to Question (\ref{CMLRradical}) relies on  algebraic properties that can be interchanged between $R/I$ and $R/J$.
The next proposition illustrates this idea and is another interesting consequence of Theorem (\ref{cddepth}).

\begin{proposition}\label{depth2}
Let $(R, I, \varphi)$ and $(R,J,\varphi)$ be cofinal flat triples, with $\sqrt{I}= \sqrt{J}$. If
$R/I$ and $R/J$ are $\varphi$-cML then $\depth R/I = \depth R/J$. In particular, $R/I$ is Cohen-Macaulay if and only if $R/J$ is Cohen-Macaulay.
\end{proposition}
\begin{proof}
    From Theorem (\ref{cddepth}), we can write
$$\depth R/I=\dim R- \cd(R,I)=\dim R-\cd(R,J)= \depth R/J,$$ which proves the claim.
\end{proof}

The situation suggested by Question (\ref{CMLRradical}) is not that simple, even when $I= \sqrt{J}$.
We borrow Example (3.12) from \cite{DDM21} to illustrate that this question is no longer affirmative in general.
\begin{example} \label{radicalrelation} There exists an ideal $I$ such that $R/I$ is neither Cohen-Macaulay nor cML, but $R/\sqrt{I}$ is both cML and Cohen-Macaulay. Indeed, take $R=k [\![ x,y,z]\!]$ and $I=(x^4,x^3y,x^2y^2z,xy^3,y^4)$. Since $$x^2y^2 \in I:_R (x,y,z),$$ we have $\depth R/I=0$. Clearly $\dim R/I=1$, so $R/I$ is not cML by Proposition (\ref{S1}) and also is not Cohen-Macaulay. On the other hand, $\sqrt{I}=(x,y)$ is a square-free monomial ideal. It then follows from Corollary (\ref{polynomials}) that $R/\sqrt{I}$ is cML and Cohen-Macaulay.
    \end{example}

In order to find an affirmative answer to Question (\ref{CMLRradical}) we shall consider rings related by means of surjective maps between their local cohomology modules. This idea has been considered in some recent investigations; see, for instance, \cite{MaSchewedeShimomoto17},   \cite{DDM21}, and references therein.


\begin{theorem} \label{surj.lc}
Let $(R, I, \varphi)$ and $(R,J,\varphi)$ be cofinal flat triples, with $\sqrt{I}= \sqrt{J}$. Assume that $I \subset J$. Suppose the natural map $$H^{i}_{\fm}(R/I) \rightarrow H^{i}_{\fm}(R/J)$$ is an epimorphism for all $i$. If $R/I$ is $\varphi$-cML,  then so is $R/J$. 
\end{theorem}
\begin{proof}
Fix an integer $i$. By applying the exact functor $\Phi^t(-)$ to the given epimorphism, we derive that the map
\begin{equation}\label{epi1}
H^{i}_{\fm}(R/\varphi^t(I)R) \rightarrow H^{i}_{\fm}(R/\varphi^{t}(J)R) 
\end{equation}
is an epimorphism for every $t$. As $R/I$ is a $\varphi$-cML, the natural map
\begin{equation}\label{epi2}
\displaystyle\vpl_t H^{i}_{\fm}(R/\varphi^t(I)R) \rightarrow H^{i}_{\fm}(R/\varphi^{s}(I)R),  
\end{equation}
is an epimorphism for all $s$. Combining (\ref{epi1}) and (\ref{epi2}) we get the following epimorphism
$$
\displaystyle\vpl_t H^{i}_{\fm}(R/\varphi^t(I)R) \rightarrow H^{i}_{\fm}(R/\varphi^{s}(J)R) \quad \mbox{for\, all} \quad s.
$$ From the assumptions, we deduce that $\{\varphi^t(I)R\}_{t \geq 0}$ and $\{\varphi^t(J)R\}_{t \geq 0}$ are both cofinal with $\{I^t\}_{t \geq 0}$. In particular, we can use (\ref{iso2}) to deduce an isomorphism
$$
\displaystyle\vpl_t H^{i}_{\fm}(R/\varphi^t(I)R) \cong \displaystyle\vpl_t H^{i}_{\fm}(R/\varphi^t(J)R).
$$
Therefore, we get an epimorphism
$$
\displaystyle\vpl_t H^{i}_{\fm}(R/\varphi^t(J)R) \rightarrow H^{i}_{\fm}(R/\varphi^{s}(J)R),
$$
for each $s$. This shows that $R/J$ is  $\varphi$-cML.\end{proof}

We note there is a connection between Theorem (\ref{surj.lc}) and cML rings of positive characteristic. For this, let us recall the following notion introduced by  Bhatt, Blickle, Lyubeznik, Singh and Zhang in \cite{LP}.

\begin{definition} A regular local ring $(R,\fm)$ of characteristic $p > 0$ is  $F$-{\it full} if the natural map $$\mathcal{F}_R (H^i_{\fm} (R)) \to H^i_{\fm} (R)$$ is surjective for all $i$. Here, $\mathcal{F}_R (-)$ denotes the Peskine-Szpiro's Frobenius functor.
Because of (\cite[lemma 2.2]{Lyubeznik2006}), this is equivalent to saying that the map $H^{i}_{\fm}(R/I^{[p]}) \rightarrow H^{i}_{\fm}(R/I)$ -- induced
from the natural projection $R/I^{[p]} \rightarrow R/I$ -- is surjective for all $i$. \end{definition}

\begin{corollary}\label{Ffull}
Let $(R, \fm)$ be a regular local ring of characteristic $p > 0$. Suppose that $R/I$ is $F$-full. If $R/I^{[p^e]}$ is cML for some $e\geq 0$, then so is $R/I$.
\end{corollary}
\begin{proof}
 By hypothesis, the map $H^{i}_{\fm}(R/I^{[p]}) \rightarrow H^{i}_{\fm}(R/I)$ is an epimorphism for all $i$.
 Applying successively the functors $\Phi, \Phi^2, \ldots, \Phi^e$, we deduce that the same is true for the map
 $$
 H^{i}_{\fm}(R/I^{[p^e]}) \rightarrow H^{i}_{\fm}(R/I).
 $$
 Now, the result follows by Theorem (\ref{surj.lc}).
\end{proof}

\subsection{Some regularity criteria in prime characteristic}
The background of the cML concept is essentially a cofinal flat triple $(R, I, \varphi)$, where $R$ is assumed to be regular.
For the positive characteristic case, $\varphi$ is taken to be the Frobenius map, whose flatness is known (by Kunz' theorem) to be equivalent to the regularity of $R$.
In this subsection, we provide a couple of (co)homological criteria for the regularity of $R$ in prime characteristic, which is, undoubtedly, a desirable property for a local ring.

Throughout this section, $(R, \mathfrak{m}, K)$ denotes a  regular local ring with ${ \rm char}\,R = p> 0$. Let $F\colon R\to R$ be the Frobenius endomorphism of $R$. For any given integer $t\geq 1$ and an $R$-module $M$, we will denote the module $\Phi^t(M)$ simply by $M^{(t)}$. Here, $\Phi$ is the functor  defined in (\ref{functorPhi}) with $\varphi = F$.  

Recall that $R$ is {\it {\rm F}-finite} if, for some (equivalently, all)  $t\geq 1$, the $R$-module $R^{(t)}$ is finite. This property is known to hold for fundamental classes
of rings, for instance, when $R$ is a complete local ring with perfect residue field or a localization of an affine algebra over a perfect field (see, e.g., \cite[p.\,398]{BH}). Note that if $R$ is {\rm F}-finite and $M$ is a finite $R$-module, then each $M^{(t)}$ is finite as well.

 Note that, as ${\rm char}\,R=p$, the local ring $R$ is equicharacteristic if ${\rm char}\,K=p$. A crucial tool to be used here is the following criterion established in \cite[Theorem 1.1 and Example 5.10(1)]{AHIY}.

\begin{lemma}\label{Av} Suppose $R$ is equicharacteristic. If $M$ is a non-zero finite $R$-module such that ${\rm pd}_RM^{(t)}<\infty$ or ${\rm id}_RM^{(t)}<\infty$ for some $t\geq 1$, then $R$ is regular.
\end{lemma}

Our first regularity criterion requires the vanishing of ${\rm Tor}_i^R(R/I, M^{(t)})$ for certain values of $i$.
\begin{theorem} Suppose $R$ is {\rm F}-finite and equicharacteristic. If $I$ is a proper ideal of $R$ with $R/I$ regular and $M$ is a non-zero finite $R$-module such that
    $${\rm Tor}_i^R(R/I, M^{(t)})=0 \quad \mbox{for\, some} \quad t\geq 1$$ and ${\rm dim}\,(R/I)+1$ consecutive positive values of $i$, then $R$ is regular.
\end{theorem}
\begin{proof}  First, for simplicity, write $S=R/I$, whose residue field is seen to be isomorphic to the residue field $K$ of $R$. Let $s={\rm dim}\,S$, which we may clearly assume to be positive.

By hypothesis, there is an integer $\ell \geq 1$ such that ${\rm Tor}_i^R(S, M^{(t)})=0$ for all $i=\ell, \ldots, \ell+s$. Since $S$ is regular, ${\rm pd}_SK=s$. Pick a minimal resolution $$0 \rightarrow G_s \to G_{s-1} \to \cdots \to G_0 \to K \rightarrow 0$$ of $K$, where the $G_i$'s are finite free $S$-modules. Write $Z_j$ for the $j$th syzygy module of $K$. Let us break the above resolution into short exact sequences  $0 \rightarrow Z_{j+1} \to G_j\to Z_j \rightarrow 0$ for $j=0, \ldots, s-1$, with $Z_0=K$ and $Z_s=G_s$. We can regard them as exact sequences of $R$-modules. The last one is $0 \rightarrow G_s \to G_{s-1}\to Z_{s-1} \rightarrow 0$. Tensoring this sequence with $M^{(t)}$ and using that ${\rm Tor}_{\ell}^R(S, M^{(t)})={\rm Tor}_{\ell +1}^R(S,M^{(t)})=0$, we obtain ${\rm Tor}_{\ell+1}^R(Z_{s-1},M^{(t)})=0$. Proceeding analogously, we derive that $${\rm Tor}_{\ell+s-1}^R(Z_1,M^{(t)})=0.$$ Also, recall  ${\rm Tor}_{\ell+s}^R(S,M^{(t)})=0$. Tensoring with $M^{(t)}$ the exact sequence $0 \rightarrow Z_1 \to G_0\to K \rightarrow 0$ yields  ${\rm Tor}_{\ell+s}^R(K,M^{(t)})=0$, which clearly forces ${\rm pd}_RM^{(t)}<\infty$. By Lemma \ref{Av}, we conclude that $R$ is regular. \end{proof}

For the second criterion, we consider the $R/I$-submodule $0:_{M^{(t)}}I=\{w\in M^{(t)} \mid Iw=0\}$.

\begin{theorem} Suppose $R$ is {\rm F}-finite and equicharacteristic. If $I$ is a proper ideal of $R$ and $M$ is a non-zero finite $R$-module such that ${\rm id}_{R/I}(0:_{M^{(t)}}I)<\infty$ {\rm (}e.g., if $R/I$ is regular{\rm )} and
 $${\rm Ext}^j_R(R/I, M^{(t)})=0 \quad \mbox{for\, some} \quad t\geq 1 \quad \mbox{and\, all} \quad j\geq 1,$$
then $R$ is regular.
\end{theorem}
\begin{proof} Pick an injective resolution $${\bf J}: 0 \rightarrow M^{(t)} \to J^{0} \to J^{1} \to \cdots $$ of $M^{(t)}$. Because  ${\rm Ext}^j_R(S, M^{(t)})=0$ for all $j\geq 1$, the co-complex ${\rm Hom}_R(S, {\bf J})$ is acyclic. Note that each ${\rm Hom}_R(S, J^{i})$ has a natural structure of $S$-module. In addition, by \cite[Lemma 3.1.6]{BH}, such $S$-modules are injective. Therefore, ${\rm Hom}_R(S, {\bf J})$ is an injective resolution of  ${\rm Hom}_R(S, M^{(t)})$ over $S$.  Applying ${\rm Hom}_S(K, -)$ to ${\rm Hom}_R(S, {\bf J})$, we get an isomorphism of co-complexes
$${\rm Hom}_S(K, {\rm Hom}_R(S, {\bf J}))\cong {\rm Hom}_R(K, {\bf J})$$ which yields isomorphisms of $R$-modules $${\rm Ext}^j_S(K, {\rm Hom}_R(S, M^{(t)}))\cong {\rm Ext}^j_R(K, M^{(t)}) \quad \mbox{for\, all} \quad j\geq 1.$$ Now, because the $S$-module $0:_{M^{(t)}}I\cong {\rm Hom}_R(S, M^{(t)})$ has finite injective dimension, we have
$${\rm Ext}^j_S(K, {\rm Hom}_R(S, M^{(t)}))=0 \quad \mbox{for\, all} \quad j\gg 0.$$ Therefore, ${\rm Ext}^j_R(K, M^{(t)})=0$ for all $j \gg 0$, which according to \cite[Proposition 3.1.14]{BH} gives ${\rm id}_RM^{(t)}<\infty$. Finally, Lemma (\ref{Av}) guarantees that $R$ is regular. \end{proof}

Now it seems natural to ask:

\begin{question}\rm Suppose $R$ is {\rm F}-finite and equicharacteristic. Let $I$ be a proper ideal of $R$ and $M$ a non-zero finite $R$-module such that ${\rm id}_{R/I}(0:_{M^{(t)}}I)<\infty$ and
 $${\rm Ext}^j_R(M^{(t)}, R/I)=0 \quad \mbox{for\, some} \quad t\geq 1 \quad \mbox{and\, all} \quad j\geq 1.$$ Is it true that $R$ must be regular?
\end{question}

\section{Another theorem, and an application to the cohomology of thickenings}

In this last section, we highlight the following result, which (together with  Theorem (\ref{cmCMLR})) will enable us to derive an application concerning the vanishing of the sheaf cohomology of thickenings in positive characteristic.

\begin{theorem}\label{main-for-app}
     Let $(R,I, \varphi)$ be a cofinal flat triple. Assume  $R$ is a regular local ring and $R/I$ is $\varphi$-cML. Fix an integer $i_0\geq 1$ and suppose $\depth (R/ \varphi^r(I)R)>i_0 +1$ for all $r$ varying in an infinite subset $\Gamma \subset \mathbb{N}$. Then, the natural map
     $$
     \displaystyle\vpl_{s} H^{i_0}_{\fm} (R/I^s) \rightarrow H^{i_0}_{\fm} (R/I^t)
     $$
     is an epimorphism for all $t$.
 \end{theorem}
 \begin{proof}
     Since $R/I$ is $\varphi$-cML, the map
     $$
     \psi_{t,i} : \displaystyle\vpl_{s} H^{i}_{\fm} (R/\varphi^s(I)R) \rightarrow H^{i}_{\fm} (R/\varphi^t(I)R)
     $$
     is surjective for all $t$ and for all $i$.
     Hence, by means of local duality, we get an injective map
     $$\psi^{\vee}_{t,i}: \Ext^{n-i}_R (R/ \varphi^t(I)R,R) \to H^{n-i}_I (R),$$
where $n = \dim R$. Since $\{\varphi^t(I)R\}_{t \geq 0}$ is cofinal with $\{I^t\}_{t \geq 0}$, for a given $t$, there exists $r \in \Gamma$ such that $\varphi^r(I)R \subset I^t$.
This inclusion induces the short exact sequence
$$
0 \to I^t/ \varphi^r(I)R \to R/\varphi^r(I)R \to R/I^t \to 0
$$
from which we derive the exact sequence
$$
\Ext^{n-i_0-1}_R (I^t/ \varphi^r(I)R,R) \to \Ext^{n-i_0}_R (R/I^t,R) \to \Ext^{n-i_0}_R (R/ \varphi^r(I)R,R). 
$$
Using the Auslander-Buchsbaum formula, 
$$
 \pd_R \left(R/\varphi^r(I)R \right)  = n- \depth R/\varphi^r(I)R < n-i_0-1.
$$
Therefore, the inequality
$$
\pd_R \left(I^t/ \varphi^r(I)R\right) \leq \pd_R \left(R/\varphi^r(I)R \right)
$$
enables us to deduce that $\Ext^{n-i_0-1}_R (I^t/ \varphi^r(I)R,R)=0$. As a consequence, the map $$\Ext^{n-i_0}_R (R/I^t,R) \to \Ext^{n-i_0}_R (R/ \varphi^r(I)R,R)$$ is injective. Composing with $\psi^{\vee}_{r,i_0}$ we obtain an injective map
$$
\phi_{t,i_0}: \Ext^{n-i_0}_R (R/I^t,R) \to H^{n-i_0}_I (R)
$$
Finally, by dualizing $\phi_{t,i_0}$, the result follows.
 \end{proof}

In the application below we are concerned with the case where $\varphi$ is the Frobenius map of a standard graded polynomial ring $R$ of prime characteristic (the graded case of Example (\ref{char p example})).

\begin{corollary}\label{mainapp} Let $k$ be a field of characteristic $p>0$ and let $R$ be a standard graded polynomial ring over $k$. Let $I$ be a non-zero homogeneous ideal of $R$ and consider the closed subscheme $X={\rm Proj}(R/I)\subset {\rm Proj}(R)$ as well as its thickenings $X_t={\rm Proj}(R/I^t)$, $t\geq 1$. Suppose $R/I$ is Cohen-Macaulay and there exists an integer  $1\leq i_0\leq {\rm dim}\,X$  such that $\depth (R/I^{[p^e]})>i_0 +1$ for infinitely many values of $e\in \mathbb{N}$. Then
    $$H^{i_0-1}(X_t, \mathcal{O}_{X_t}(-\ell))=0 \quad \mbox{for\, all} \quad \ell >0 \quad \mbox{and\, all} \quad t.$$
\end{corollary}
\begin{proof} First, letting $\fm$ be the homogeneous maximal ideal of $R$, it is well-known that $$H^{i_0}_{\fm} (R/I^t)_j\cong H^{i_0-1}(X_t, \mathcal{O}_{X_t}(j)) \quad \mbox{for\, all} \quad j\in {\mathbb Z}$$ (resp.\, for all $j<0$) whenever $i_0\geq 2$ (resp.\, $i_0=1$). By  Theorem (\ref{cmCMLR}), the ring $R/I$ is cML. Now, Theorem (\ref{main-for-app}) (which is seen to admit an analogous graded version) yields that the natural map
     $$
     \displaystyle\vpl_{s} H^{i_0}_{\fm} (R/I^s) \rightarrow H^{i_0}_{\fm} (R/I^t)
     $$
     is a (degree zero) epimorphism for all $t$. Thus, in each degree $j$, there is an epimorphism $$\displaystyle\vpl_{s} \left[H^{i_0}_{\fm} (R/I^s)\right]_j
     =\left[\displaystyle\vpl_{s} H^{i_0}_{\fm} (R/I^s)\right]_j \rightarrow H^{i_0}_{\fm}(R/I^t)_j
     $$ Using these facts, we get, for all $t$, a  vector space epimorphism 
     $$
     \displaystyle\vpl_{s} H^{i_0-1}(X_s, \mathcal{O}_{X_s}(j)) \rightarrow H^{i_0-1}(X_t, \mathcal{O}_{X_t}(j)).
     $$ On the other hand, \cite[Corollary 2.4]{LP} gives $
     \displaystyle\vpl_{s} H^{i_0-1}(X_s, \mathcal{O}_{X_s}(j))=0
     $ for all $j<0$. The result follows. \end{proof}

We close the paper with illustrations of Corollary (\ref{mainapp}) (in low characteristic, for computational reasons).

\begin{example}\rm Let $R$ be a standard graded polynomial ring in 8 indeterminates over a field $k$ of characteristic $p=2$ or $p=3$. Let $X\subset {\mathbb P}_k^7$ be the closed projective subscheme defined by the ideal $I$ of $R$ generated by the maximal minors of a $2\times 4$ generic matrix. Then, $R/I$ is Cohen-Macaulay, as ${\rm dim}\,R/I=5={\rm depth}\,R/I$. Further computations show that $${\rm depth}\,R/I^{[p^e]}=5 \quad \mbox{for\, all} \quad e\geq 0.$$ So we can apply Corollary (\ref{mainapp}) with $i_0=3$ (note $i_0<4={\rm dim}\,X$) in order to get
$$H^{2}(X_t, \mathcal{O}_{X_t}(-\ell))=0 \quad \mbox{for\, all} \quad \ell >0 \quad \mbox{and\, all} \quad t.$$ Finally, it is worth observing that, for all $t\geq 2$, $$H^{3}_{\fm} (R/I^t)\neq 0$$ since ${\rm depth}\,R/I^t=3$. As a last side remark, we can even guess (due to the genericity of the matrix) that in this example the depth of $R/I^{[p^e]}$ should still be 5 for any $p\geq 5$; this holds, for instance, if ($e=1$ and) $p\in \{5, 7, 11, 13\}$.

\end{example}

\begin{example} A first non-generic instance follows easily from a slight modification of the previous example.  Let $p=2$ or $p=3$, and $X\subset {\mathbb P}_k^6={\rm Proj}(R)={\rm Proj}(k[x_0, x_1, x_2, x_3, x_4, x_5, x_6])$ be the closed projective subscheme defined by the ideal
$$I=I_2\left(\begin{array}{cccc}
x_0  & x_1 & x_2 & x_3\\
x_4 & x_5 & x_6 & x_0 
\end{array}\right) \subset R.$$ Note that $R/I$ is Cohen-Macaulay of dimension 4. Also, ${\rm depth}\,R/I^{[p^e]}=4$ for all $e$. Thus, applying Corollary (\ref{mainapp}) with $i_0=2$, we obtain    $$H^{1}(X_t, \mathcal{O}_{X_t}(-\ell))=0 \quad \mbox{for\, all} \quad \ell >0 \quad \mbox{and\, all} \quad t.$$ Finally, we observe that $$H^{2}_{\fm} (R/I^t)\neq 0$$ as ${\rm depth}\,R/I^t=2$, for all $t\geq 2$.
\end{example}

\begin{example}  Let $p=2$ and $X\subset {\mathbb P}_k^8={\rm Proj}(R)={\rm Proj}(k[x_0, x_1, x_2, x_3, x_4, x_5, x_6, x_7, x_8])$ be the closed projective subscheme defined by the ideal
$$I=I_3\left(\begin{array}{cccc}
x_0  & x_1 & x_2 & x_3\\
x_4 & x_5 & x_6 & x_7\\ 
x_8  & x_0 & x_1 & x_2
\end{array}\right) \subset R.$$ Here, the ideal $I$ is perfect of height 2, so $R/I$ is Cohen-Macaulay of dimension 7. It can be checked that ${\rm depth}\,R/I^{[2^e]}=7$ for all $e$. Thus, applying Corollary (\ref{mainapp}) first with $i_0=4$, we obtain    $$H^{3}(X_2, \mathcal{O}_{X_2}(-\ell))=0 \quad \mbox{for\, all} \quad \ell >0.$$ Furthermore, we notice that $$H^{4}_{\fm} (R/I^2)\neq 0$$ as ${\rm depth}\,R/I^2=4$ (while ${\rm depth}\,R/I^t=5$ for all $t\geq 3$). In addition, we are also allowed to take $i_0=5$ and conclude that
$$H^{4}(X_t, \mathcal{O}_{X_t}(-\ell))=0 \quad \mbox{for\, all} \quad \ell >0 \quad \mbox{and\, all} \quad t.$$
\end{example}

\bigskip

\noindent{\bf Acknowledgments.} 
 The second-named author was in part supported by a grant from IPM (No. 1404130017). The third-named author was partially supported by CNPq (grants 406377/2021-9, 313357/2023-4 and 408698/2023-3).

\bigskip

\noindent{\bf Data availability statement.} No datasets were generated or analysed during the current study.

\bigskip

\noindent{\bf Conflict of interest statement.} On behalf of all authors, the corresponding author states that there is no conflict of interest.

\
\bibliographystyle{alpha}
\bibliography{sample.bib}

\end{document}